\documentclass[12pt]{article}
\usepackage{amsfonts}
\usepackage[all]{xy}
\newtheorem{lemma}{Lemma}
\newtheorem{theorem}{Theorem}
\newtheorem{proposition}{Proposition}
\newtheorem{Definition}{Definition}
\newenvironment{proof}{{\bf Proof}.}{\hfill $\Box$}

\begin{document}
\title{On the multi-dimensional Favard Lemma}
\author{Abdallah Dhahri\\\vspace{-2mm}\scriptsize Department of Mathematics \\\vspace{-2mm}\scriptsize Faculty of Sciences of Tunis \\\vspace{-2mm}\scriptsize University of Tunis El-Manar, 1060 Tunis, Tunisia\\\vspace{-2mm}\scriptsize abdallah.dhahri@fst.rnu.tn\\ \\
Ameur Dhahri\\\vspace{-2mm}\scriptsize Volterra Center, University of Roma Tor Vergata\\\vspace{-2mm}\scriptsize Via Columbia 2, 00133 Roma, Italy\\\scriptsize ameur@volterra.uniroma2.it}
\date{}
\maketitle
\begin{abstract}
We prove that the creator operators, on the $d$ commuting indeterminates polynomial algebra, are linearly independent. We further study the connection between the classical (one dimensional) and the multi-dimensional ($d$-dimensional, $d\geq 1$) Favard Lemmas. Moreover, we investigate the dependence of the Jacobi sequences on the linear change of basis of $\mathbb C^d$. Finally we prove that the Jacobi sequences associated to a product probability measure on $\mathbb{R}^d$, with finite moments of any order, are diagonal matrices in the basis introduced by the tensor product of the orthogonal polynomials of the factor measures (see page 21).
\end{abstract}
\section{Introduction}
The classical (or one-dimensional) Favard Lemma (cf \cite{ch1978}, \cite{Is-As}, \cite{sz1975}) states that to any probability measure $\mu$ on the real line with finite moments of any order it associated two sequences, called {\it Jacobi sequences} of $\mu$,
 $$\left( (w_n)_{n\in\mathbb{N}},\:\:(\alpha_n)_{n\in\mathbb{N}}\right),\:\:\:w_n\in\mathbb{R}_+,\;\alpha_n\in\mathbb{R}$$
such that
$$w_n=0\Rightarrow w_{n+k}=0,\:\forall k\in\mathbb{N}$$
and conversely, given two such sequences, it is associated:
\begin{enumerate}
\item[(i)] a state, induced by a probability measure on $\mathbb{R}$, on the one indeterminate polynomial algebra $\mathcal{P}$,
\item[(ii)] an orthogonal decomposition of $\mathcal{P}$ canonically associated to this state.
\end{enumerate}

In the multi-dimensional case (cf \cite{DX},\cite{K1}, \cite{K2},\cite{X}) the formulations of these results are recently given by identifying the theory of multi-dimensional orthogonal polyomials with the theory of symmetric interacting Fock spaces (cf \cite{AcBaDh}). The multi-dimensional analogue of positive numbers $w_n$ (resp. real numbers $\alpha_n$) are the positive definite matrices (resp. Hermitean matrices).

In this paper we prove the injectivity and the linear independence properties of the creator operators. We show that the Jacobi sequences depend on the choice of basis of $\mathbb{C}^d$. We study the relation between the classical and $d$-dimensional Favard Lemmas in the case of $d=1$. Moreover, we prove that the Jacobi sequences, in the case of probability measure product, are diagonal matrices in the basis introduced by the tensor product of the orthogonal polynomials of the factor measures.

This paper is organized as follows. In section 2, we recall from [1] the basic properties of the complex polynomial algebra in $d$ commuting indeterminates and  the multi-dimensional Favard Lemma. The dependence of the Jacobi sequences on the choice of a basis of $\mathbb{C}^d$ is treated in section 3. In section 4, we discuss the connection between the classical and multi-dimensional Favard Lemmas ($d=1$). Finally, in section 5 we give the explicit forms of the Jacobi sequences in the case of product probability measures on $\mathbb{R}^d$ with finite moments of any order.

\section{The multi-dimensional Favard Lemma}
In this section, we recall the basic properties of the polynomial algebra in $d$ commuting indeterminates and the multi-dimensional Favard Lemma. We refer the interested reader to \cite{AcBaDh} for more details.

\subsection{The polynomial algebra in $d$ commuting indeterminates}
Let $d\in \mathbb{N}^*=\mathbb{N}\ \{0\}$ and let
$$\mathcal{P}=\mathbb{C}[(X_j)_{1\leq j\leq d}]$$
be the complex polynomial $*$-algebra in the commuting indeterminates $(X_j)_{1\leq j\leq d}$ with the involution uniquely determined by the prescription that the $X_j$ are self-adjoint. For all $v=(v_1,\dots,v_d)\in\mathbb{C}^d$ denote
$$X_v:=\sum_{j=1}^dv_jX_j$$

A {\it monomial} of degree $n\in\mathbb{N}$ is by definition any product of the form
$$M:=\prod_{j=1}^d X_j^{n_j}$$
where, for any $1\leq j\leq d$, $n_j\in\mathbb{N}$ and $n_1+\dots+n_d=n$.

Denote by $\mathcal{P}_{n]}$ the vector subspace of $\mathcal{P}$ generated by the set of monomials of degree less or equal than $n$. It is clear that
$$\mathcal{P}=\bigcup_{n\in\mathbb{N}}\mathcal{P}_{n]}$$
\begin{Definition}
For $n\in\mathbb{N}$, a subspace $\mathcal{P}_n\subset\mathcal{P}_{n]}$ is monic of degree $n$ if
$$\mathcal{P}_{n]}=\mathcal{P}_{n-1]}\dot+\mathcal{P}_n$$
(with the convention $\mathcal{P}_{-1]}=\{0\}$ and where $\dot+$ means a vector space direct sum) and $\mathcal{P}_n$ has a linear basis $\mathcal{B}_n$ with the property that for each $b\in\mathcal{B}_n$, the highest order term of $b$ is a single multiple of a monomial of degree $n$. Such a basis is called a perturbation of the monomial basis of order $n$ in the coordinates $(X_j)_{1\leq j\leq d}$.
\end{Definition}

Note that any state $\varphi$ on $\mathcal{P}$ defines a pre-scalar product
\begin{eqnarray*}
\langle.,.\rangle_\varphi:\mathcal{P}\times\mathcal{P}&\rightarrow& \mathbb{C}\\
(a,b)&\mapsto&\langle a,b\rangle_\varphi=\varphi(a^*b)
\end{eqnarray*}
with $\langle1_\mathcal{P},1_\mathcal{P}\rangle_\varphi=1$, where $1_\mathcal{P}=1$.
\begin{lemma}\label{state}
Let $\varphi$ be a state on $\mathcal{P}$ and denote $\langle\ \cdot,\cdot \ \rangle=\langle\ \cdot,\cdot \ \rangle_\varphi$ be the associated pre-scalar product. Then there exists a unique gradation
\begin{equation}\label{phi-gradation}
\mathcal{P}=\bigoplus_{n\in\mathbb{N}}\mathcal{P}_{n,\varphi}
\end{equation}
called the $\varphi$-orthogonal polynomial decomposition of $\mathcal{P}$, with the following properties:
\begin{enumerate}
\item[(i)] (\ref{phi-gradation}) is orthogonal for the pre-scalar product $\langle\ \cdot,\cdot\ \rangle$ on $\mathcal{P}$, i.e.
\begin{eqnarray*}
\mathcal{P}_{m,\varphi}\bot\mathcal{P}_{n,\varphi}, \qquad \forall m\neq n
\end{eqnarray*}
\item[(ii)] (\ref{phi-gradation}) is compatible with the filtration $(\mathcal{P}_{n]})_n$ in the sense that
\begin{equation}\label{*}
\mathcal{P}_{n]}=\bigoplus_{h=0}^n\mathcal{P}_{h,\varphi},\qquad \forall n\in\mathbb{N},
\end{equation}
\item[(iii)] for each $n\in\mathbb{N}$ the space $\mathcal{P}_{n,\varphi}$ is monic.
\end{enumerate}

Conversely, let be given:
\begin{enumerate}
\item[(j)] a vector space direct sum decomposition of $\mathcal{P}$
\begin{equation}\label{direct-sum}
\mathcal{P}=\sum_{n\in\mathbb{N}}^{\cdot}\mathcal{P}_n
\end{equation}
is compatible with the filtration in the sens of (\ref{*}) and such that $\mathcal{P}_0=\mathbb{C}.1_\mathcal{P}$, and for each $n\in\mathbb{N}$, $\mathcal{P}_n$ is monic of degree $n$,
\item[(jj)] for all $n\in\mathbb{N}$ a pre-scalar product $\langle \ \cdot,\cdot \ \rangle_{n}$ on $\mathcal{P}_n$ with the property that $1_\mathcal{P}$ has norm $1$ and the unique pre-scalar product $\langle \ \cdot,\cdot \ \rangle$ on $\mathcal{P}$ defined by the conditions:
\begin{eqnarray*}
\langle\ \cdot,\cdot\ \rangle|_{\mathcal{P}_n}&=&\langle \ \cdot,\cdot \ \rangle_{\mathcal{P}_n},\qquad \forall n\in\mathbb{N}\\
\mathcal{P}_{m}&\bot&\mathcal{P}_{n}, \qquad \forall m\neq n
\end{eqnarray*}
is such that the multiplication operators by the coordinates $X_j$ ($1\leq j\leq d$) are $\langle\ \cdot,\cdot\ \rangle$-symmetric.\smallskip
\end{enumerate}
Then, there exists a state $\varphi$ on $\mathcal{P}$ such that the decomposition (\ref{direct-sum}) is the orthogonal polynomial decomposition of $\mathcal{P}$ with respect to $\varphi$.
\end{lemma}

\subsection{The symmetric Jacobi relations and the CAP operators}
In the following, we fix a state $\varphi$ on $\mathcal{P}$ and we follow the notations of Lemma \ref{state} with the exception that we omit the index $\varphi$. We write $\langle.,.\rangle$ for the pre-scalar product $\langle.,.\rangle_\varphi$, $\mathcal{P}_k$ for the space $\mathcal{P}_{k,\varphi}$ and $P_{k]}:\mathcal{P}\rightarrow \mathcal{P}_{k]}$ for the $\langle.,.\rangle$-orthogonal projector in the pre-Hilbert space sense (see \cite{AcBaDh} for more details). Put
$$P_n=P_{n]}-P_{n-1]}$$
It is clear that $P_n=P_n^*$, $\sum_{n\geq0}P_n=1$ and $P_nP_m=\delta_{nm}P_n$ for all $n,m\in\mathbb{N}$. Moreover, for any $1\leq j\leq d$ and any $n\in\mathbb{N}$, one has
\begin{eqnarray}\label{Jacobi relation}
X_jP_n=P_{n+1}X_jP_n+P_nX_jP_n+P_{n-1}X_jP_n
\end{eqnarray}
with the convention that $P_{-1]}=0$. The identity (\ref{Jacobi relation}) is called the {\it symmetric Jacobi relation}.

Now, for each $1\leq j\leq d$ and $n\in\mathbb{N}$ we define the operators $a^\varepsilon_{j|n},$ $\varepsilon\in\{+,0,-\}$, as follows:
\begin{eqnarray}
a^{+}_{j|n}&:=&P_{n+1}X_jP_{n}\Big|_{{{\cal P}}_{n}}
\ : \ {{\cal P}}_{n}\longrightarrow {{\cal P}}_{n+1}\label{1df-a+j|n}\nonumber,\\
a^{0}_{j|n}&:=&P_{n}X_jP_{n}\Big|_{{{\cal P}}_{n}}
\ : \ {{\cal P}}_{n}\longrightarrow {{\cal P}}_{n}\label{cre-ann-con},\\
a^{-}_{j|n}&:=&P_{n-1}X_jP_{n}\Big|_{{{\cal P}}_{n}}
\ : \ {{\cal P}}_{n}\longrightarrow {\cal P}_{n-1}\label{1df-a-j|n}.\nonumber
\end{eqnarray}

Note that in this context, the sum
\begin{equation}\label{weak-sum}
{{\cal P}}=\bigoplus_{n\in{\mathbb N}}{{\cal P}}_{n}
\end{equation}
is orthogonal and meant in the weak sense, i.e. for each
element $Q\in {{\cal P}}$ there is a finite set
$I\subset{\mathbb N}$ such that
\begin{equation}\label{finite-decomposition}
    Q=\sum_{n\in I}p_n
\; ,\qquad p_n\in {{\cal P}}_{n}
\end{equation}

\begin{theorem}\label{th-Q--dec}
On ${{\cal P}}$, for any $1\leq j\leq d $, the following operators are well defined
\begin{eqnarray*}\label{1df-a+j}
a^{+}_j&:=&\sum_{n\in{\mathbb N}}a^{+}_{j|n}\\
a^{0}_j&:=&\sum_{n\in{\mathbb N}}a^{0}_{j|n}\\
a^{-}_j&:=&\sum_{n\in{\mathbb N}}a^{-}_{j|n}
\end{eqnarray*}
and one has
\begin{equation}\label{q-dec-Xj}
X_j=a^{+}_j+a^{0}_j+a^{-}_j
\end{equation}
in the sense that both sides of (\ref{q-dec-Xj})
are well defined on ${{\cal P}}$ and the equality holds.
\end{theorem}
Identity (\ref{q-dec-Xj}) is called the {\it quantum decomposition} of the variable $X_j$.
\begin{proposition}
For any $1\leq j\leq d$ and $n\in\mathbb N$, one has
\begin{eqnarray}\label{cap}
(a^{+}_{j|n})^{*} &=& a^{-}_{j|n+1}\qquad; \qquad
(a^{+}_{j})^{*} = a^{-}_{j}\label{a+j*=a-j}\nonumber\\
(a^{0}_{j|n})^{*} &=& a^{0}_{j|n}\qquad; \qquad
(a^0_{j})^*=a^0_{j}\label{a0j*=a0-j}
\end{eqnarray}
Moreover, for each $j,k\in\{1,\dots,d\}$, one has
$$[a^+_j,a^+_k]=0$$
\end{proposition}

Let $V$ be a complex Hilbert vector space of dimension $d$. Given a basis $(e_j)_{1\leq j\leq d}$ of $V$, the operators (\ref{cap}) allow to define two complex linear maps
$$a^{\varepsilon}:\;v=\sum_{j=1}^db_je_j\in (V,\:(e_j)_{1\leq j\leq d})\rightarrow a^\varepsilon_v=\sum_{j=1}^db_ja_j^\varepsilon \in\mathcal{L}(\mathcal{P}),$$
where $\varepsilon\in\{0,+\}$ and $\mathcal{L}(\mathcal{P})$ is the vector space of the linear maps of $\mathcal{P}$ into itself. Then, one defines
$$a^-_v:=(a^+_v)^*.$$
In particular, we have
$$a^\eta_{e_j}=a^\eta_j,$$
where $\eta\in\{+,0,-\}.$ For all $n\in\mathbb{N}$ and $v=\sum_{j=1}^db_je_j$, we denote
$$X_v:=\sum_{j=1}^db_jX_j.$$
Moreover, if $\eta\in\{+,0,-\}$, we denote
$$a^\eta_v\Big|_{\mathcal{P}_n}:=a^\eta_{v|n}:=\sum_{j=1}^db_ja^\eta_{j|n},$$
it is clear that $a^\eta_v=\sum_{n\geq0}a^\varepsilon_{v|n}$.

\subsection{The  multi-dimensional Favard Lemma}

Recall that

$$\mathcal{P}_n:=\{a^+_{v|n}(\mathcal{P}_{n-1});\:\:v\in V\}.$$


Now, denote $\otimes$ the algebraic tensor product and $\hat{\otimes}$ its symmetrization. The tensor algebra over $V$ is the vector space
$$\mathcal{T}(V):=\sum^{\cdot}_{n\in\mathbb{N}}V^{\otimes n}$$
with multiplication given by
$$(u_n\otimes\dots\otimes u_1)\otimes(v_n\otimes\dots\otimes v_1):=u_n\otimes\dots\otimes u_1\otimes v_n\otimes\dots\otimes v_1$$
for all $n,m\in\mathbb{N}$ and all $u_j,v_j\in V$. The $*$-sub-algebra of $\mathcal{T}(V)$ generated by the elements of the form
$$v^{\otimes n}:=v\otimes\dots\otimes v\:(n-times),\:\forall n\in\mathbb{N},\;\forall v\in V$$
is called the {\it symmetric tensor algebra} over $V$ and denoted $\mathcal{T}_{sym}(V)$.

\begin{lemma}\label{id-symm-tens}

For all $n\in\mathbb{N}^*$, let ${\cal P}_{n}$ be defined as above. Denoting, for $v\in V$, $a^+_v:=\sum_{n\in\mathbb{N}}a^+_{v|k}$ and $\Phi=1_\mathcal{P}$. Then the map
\begin{equation}\label{symm-tens-vn-a+vn}
 U_n:\;v_n\hat\otimes v_{n-1}\hat\otimes\cdots \hat\otimes v_{1}
 \in V^{\hat \otimes n} \ \longmapsto \
a^+_{v_n}a^+_{v_{n-1}} \cdots a^+_{v_{1}}\Phi
\in {\cal P}_{n},
\end{equation}
extends uniquely to a vector space isomorphism with the property that for all $v\in V$ and $\xi_{n-1}\in V^{\hat \otimes (n-1)}$
$$U_n(v\hat\otimes\xi_{n-1})=a^+_{v}U_{n-1}\xi_{n-1}$$

For $n=0$ we put
$$U_0:z\in\mathbb{C}:=V^{\hat\otimes0}\longmapsto U_0(z):=z\in\mathbb{C}1_\mathcal{P}\in\mathcal{P}_0.$$
\end{lemma}

The multi-dimensional Favard Lemma is given by the following theorem.
\begin{theorem}\label{multi}
Let $\mu$ be a probability measure on ${\mathbb R}^d$
with finite moments of all orders and denote $\varphi$ the state on $\mathcal{P}$ given by
$$\varphi(b)=\int_{\mathbb{R}^d}b(x_1,\dots,x_d)d\mu(x_1,\dots,x_d),\:b\in\mathcal{P}$$
Then, there exist two sequences
$$
(\Omega_n)_{n\in{\mathbb N}}\qquad ;\qquad
(\alpha_{.|n})_{n\in{\mathbb N}}
$$
satisfying:
\begin{enumerate}
\item[(i)] for all $n\in{\mathbb N}, \Omega_n$ is a linear operator on $V^{\hat\otimes n}$ positive and symmetric with respect to the tensor scalar product given by
\begin{equation}\label{scalar}
\langle u^{\otimes n},v^{\otimes m}\rangle_{V^{\hat{\otimes}n}}:=\delta_{m,n}\langle u,v\rangle_{V}^n;\:\:
\forall u,v\in V; \forall n\in{\mathbb N}
\end{equation}
where $\langle\ \cdot,\cdot\ \rangle_{V}$ is a scalar product on $V$.
\item[(ii)] denoting for all $n\in\mathbb{N}$
\begin{equation}\label{6.1}
\langle\xi_n,\eta_n\rangle_n:=\langle\xi_n,\Omega_n\eta_n\rangle_{V^{\hat{\otimes}n}};\:\:\:\xi_n,\eta_n\in V^{\hat\otimes n}
\end{equation}
the pre-scalar product on $V^{\hat{\otimes}n}$ defined by $\Omega_n$ and $| \ \cdot,\cdot \ |_n$ the associated pre-norm. For all $n\in\mathbb{N},v\in V$ and $\eta_{n-1}\in V^{\hat{\otimes}(n-1)}$, one has
$$|\eta_{n-1}|_{n-1}=0\Rightarrow |v\hat\otimes \eta_{n-1}|_n=0$$

\item[(iii)] for all $n\in{\mathbb N},$
$$\alpha_{.|n} \ : \ v\in V \ \to \ \alpha_{v|n}\in\mathcal{L}\Big(V^{\hat{\otimes}n}\Big)
$$
is a linear map and for all $v\in V$ such that $v=\bar{v}$, $\alpha_{v|n}$ is
a linear operator on $V^{\hat{\otimes}n}$,
symmetric for the pre-scalar product $\langle \ \cdot  ,  \cdot \ \rangle_n$ on $V^{\hat{\otimes}n}$;
\item[(iv)] the sequence $\Omega_n$ defines a symmetric interacting Fock space struture over $V$ endowed with the scalar product $\langle\,.\:,\:.\,\rangle_V$ and the operator
\begin{equation}\label{gradation-preserving}
\!\!\!\!\!\!\!\!\!\!\!\!\!U:=\bigoplus_{k\in\mathbb{N}}U_k:\bigoplus_{k\in\mathbb{N}}\left(
V^{\hat{\otimes}k} ,  \ \langle  \cdot  , \cdot  \rangle_k\right)\rightarrow\bigoplus_{k\in\mathbb{N}}\left(\mathcal{P}_k ,  \ \langle  \cdot  , \cdot  \rangle_{\mathcal{P}_k}\right)=\left(\mathcal{P},\ \langle  \cdot  , \cdot \rangle\right)
\end{equation}
is an orthogonal gradation preserving unitary isomorphism of pre-Hilbert spaces, where $\langle  \cdot  , \cdot  \rangle_{\mathcal{P}_k}$ is the pre-scalar product induced by $\varphi$ on $\mathcal{P}_k$.
\end{enumerate}
Moreover, denoting
\begin{eqnarray}\label{FAv-IFS}
\Gamma\left(V,\;(\Omega_n)_n\right)
:=
\bigoplus_{n\in{\mathbb N}}\left(
(V^{\hat{\otimes}n} \ ,  \
\langle \ \cdot  ,  \cdot \
\rangle_n\right)
\end{eqnarray}
the symmetric interacting Fock space defined by the sequence
$(\Omega_n)_{n\in{\mathbb N}}$, $A^\pm$ the creation and annihilation fields associated to it, $P_{\Gamma,n}$ the projection onto the $n-th$ space of the gradation (\ref{FAv-IFS}), and $N$ the number operator associated to this gradation i.e.
$$N:=\sum_{n\in\mathbb{N}}nP_{\Gamma,n},$$
the gradation preserving unitary pre-Hilbert space isomorphism (\ref{gradation-preserving}) satisfies
\begin{eqnarray*}
U\Phi&=&1_\mathcal{P}\\
U^{-1}X_vU&=&A^{+}_v+\alpha_{v,N}+A^{-}_v,\:\:\forall v=\bar{v}\in V,
\end{eqnarray*}
where $\alpha_{v,N}$ is the symmetric operator defined by
$$\alpha_{v,N}:=\sum_{n\in\mathbb{N}}\alpha_{v|n}P_{\Gamma,n}.$$

Conversely, given two sequences $(\Omega_n)_{n\in{\mathbb N}}$
and $(\alpha_{.|n})_{n\in{\mathbb N}}$ satisfying (i), (ii), (iii) and (iv) above, there exists a state $\varphi$ on $\mathcal{P}$, such that for any probability measure $\mu$ on ${\mathbb R}^d$, inducing the state $\varphi$ on $\mathcal{P}$, the pair of sequences\\$\left((\Omega_n)_{n\in\mathbb{N}},\:(\alpha_{.|n})_{n\in\mathbb{N}}\right)$ is the one associated to $\mu$ according to the first part of the theorem.
\end{theorem}

{\bf Remark}: From the proof of the above theorem (cf \cite{AcBaDh}) one has
\begin{equation}\label{alphan}
\alpha_{.|n}=U_n^{-1}a^0_{.|n}U_n
\end{equation}
\begin{Definition}
The sequences $(\Omega_n)_n$ and $(\alpha_{.|n})_n$ in Theorem \ref{multi} are called Jacobi sequences associated to the probability measure $\mu$.
\end{Definition}

\section{Injectivity of the creator opertaors}
In this section we give some properties associated to the creator operators.
\begin{theorem}
If $v\neq 0_{V}$, the creator operator $a^+_v$ is an injective operator, i.e:
$$a_v^+\xi_n=0 \Rightarrow \xi_n=0$$
\end{theorem}
\begin{proof}
Let $v\neq 0_{V}=\sum_{j=1}^db_je_j$ and $\xi_n\in\mathcal{P}_n$ such that
\begin{equation}\label{eq1}
a_v^+\xi_n=0
\end{equation}
 Since $\xi_n\in\mathcal{P}_n$ then $\xi_n$ is of the form
$$\xi_n=\sum_{n_1+\dots+n_d\leq n}c_{n_1,\dots,n_d}X_1^{n_1}\dots X_d^{n_d}$$
From the quantum decomposition it follows that
\begin{eqnarray*}
a_v^+\xi_n=X_v\xi_n-a^0_v\xi_n-a^-_v\xi_n=0
\end{eqnarray*}
Hence one gets
\begin{eqnarray}\label{eq2}
&&X_v\sum_{n_1+\dots+n_d= n}c_{n_1,\dots,n_d}X_1^{n_1}\dots X_d^{n_d}+X_v\sum_{n_1+\dots+n_d\leq n-1}c_{n_1,\dots,n_d}X_1^{n_1}\dots X_d^{n_d}\nonumber\\
&&\:\:\:-a^0_v\xi_n-a^-_v\xi_n=0
\end{eqnarray}
Note that
\begin{eqnarray}
&&X_v\sum_{n_1+\dots+n_d= n}c_{n_1,\dots,n_d}X_1^{n_1}X_1^{n_1}\dots X_d^{n_d}\in\mathcal{P}_{n+1}\label{eq3}\\
&&X_v\sum_{n_1+\dots+n_d\leq n-1}c_{n_1,\dots,n_d}X_1^{n_1}X_1^{n_1}\dots X_d^{n_d}-a^0_v\xi_n-a^-_v\xi_n\in\mathcal{P}_{n]}\label{eq3'}
\end{eqnarray}
Therefore identities (\ref{eq2}), (\ref{eq3}) and (\ref{eq3'}) imply that
\begin{eqnarray}\label{eq4}
0&=&\sum_{1\leq j\leq d}\sum_{n_1+\dots+n_d= n}X_jX_1^{n_1}\dots X_d^{n_d}\Big(b_jc_{n_1,\dots,n_d}\Big)\nonumber\\
&=&\sum_{n_1+\dots+n_d=n}X_1^{n_1+1}X_2^{n_2}\dots X_d^{n_d}\Big(b_1c_{n_1,\dots,n_d}\Big)\nonumber\\
&&+\sum_{n_1+\dots+n_d=n}X_1^{n_1}X_2^{n_2+1}\dots X_d^{n_d}\Big(b_2c_{n_1,\dots,n_d}\Big)\\
&&\;\vdots\nonumber\\
&&+\sum_{n_1+\dots+n_d=n}X_1^{n_1}X_2^{n_2}\dots X_d^{n_d+1}\Big(b_dc_{n_1,\dots,n_d}\Big)\nonumber
\end{eqnarray}
 Note that $v\neq 0_{V}$. Then, there exists $i_0\in\{1,\dots,d\}$ such that $b_{i_0}\neq0$. Without loss of generality, suppose that $i_0=1$ and $b_1\neq0$. In the following by using the induction on the degree of the indeterminate $X_1$ in (\ref{eq4}), we will prove that for all $n_1,\dots,n_d$ such that $n_1+\dots+n_d=n$ one has $c_{n_1,\dots,n_d}=0$.
\begin{enumerate}
\item[(i)] From (\ref{eq4}) one has
$$0=b_1c_{n,0,\dots,0}$$
which is the coefficient of $X^{n+1}$ in (\ref{eq4}). This implies that $c_{n,0,\dots,0}=0$. In the same way, one has
$$0=b_1c_{n-1,0,\dots,0,1,0,\dots,0}+b_jc_{n,0,\dots,0}\:\:(j\geq2,\:n_j=1)$$
which is the coefficient of $X_1^nX_j$ in (\ref{eq4}). Since $c_{n,0,\dots,0}=0$ and $b_1\neq0$ then one has $c_{n-1,0,\dots,0,1,0,\dots,0}=0$. Similarly, one has
$$0=b_1c_{n-2,0,\dots,0,2,0,\dots,0}+b_jc_{n-1,0,\dots,1,0,\dots,0}\:\:(j\geq 2)$$
which is the coefficient of $X_1^{n-1}X_j^2$ in (\ref{eq4}). It follows that $c_{n-2,0,\dots,0,2,0,\dots,0}=0$.

Finally, using the same raisonment, we prove that
$$c_{n_1,0,\dots,0,n_j,0,\dots,0}=0$$
for all $n_1,n_j$ such that $n_1+n_j=n$.
\item[(ii)] Now, let $k\in\{2,\dots,d-1\}$ and suppose that $c_{n_1,0,\dots,0,n_{i_1},0,\dots,0,n_{i_k},0,\dots,0}=0$ for all $n_1,n_{i_1},\dots,n_{i_k}$ such that $n_1+n_{i_1}+\dots+n_{i_k}=n$ with $i_m\in\{2,\dots,d\}$ for all $m=1,\dots,k$. Note that, by identification of the coefficient of $X_1^{n-k}X_{i_1}\dots X_{i_k}X_{i_{k+1}}$ in (\ref{eq4}), one gets
\begin{eqnarray}\label{eq5}
0\!\!\!&=&\!\!\!b_1c_{n-k-1,0,\dots,0,1,0,\dots,0,1,0,\dots,0,1,0,\dots,0}\:\:(n_{i_m}=1 \mbox{ with } 1\leq m\leq k+1)\\
&&+b_{i_1}c_{n-k,0,\dots,0,1,0,\dots,0,1,0,\dots,0}\:\:(n_{i_1}=0,\;n_{i_m}=1 \mbox{ with } 2\leq m\leq k+1)\nonumber\\
&&\vdots\nonumber\\
&&+b_{i_{k+1}}c_{n-k,0,\dots,0,1,0,\dots,0,1,0,\dots,0}\:\:(n_{i_{k+1}}=0,\;n_{i_m}=1 \mbox{ with } 1\leq m\leq k)\nonumber
\end{eqnarray}
By using the induction assumption and the fact that $b_1\neq0$, identity (\ref{eq5}) implies that
 $$c_{n-k-1,0,\dots,0,1,0,\dots,0,1,0,\dots,0,1,0,\dots,0}=0\:\:(n_{i_m}=1 \mbox{ with } 1\leq m\leq k+1)$$
Similarly, by identification of the coefficient of $X_1^{n-k-2}X_{i_1}^2\dots X_{i_k}X_{i_{k+1}}$ in (\ref{eq4}), we show that
 $$c_{n-k-2,0,\dots,0,n_{i_1},0,\dots,0,n_{i_{k+1}},0,\dots,0}=0$$
where $n_{i_1}=2 \mbox{ and } n_{i_m}=1\mbox{ with } 2\leq m\leq k+1$. We continuate the same raisonment as above we show that
$$c_{n_1,0,\dots,0,n_{i_1},0,\dots,0,n_{i_{k+1}},0,\dots,0}=0$$
for all $n_1,n_{i_1},\dots,n_{i_{k+1}}$ such that $n_1+n_{i_1}+\dots+n_{i_{k+1}}=n$.
\end{enumerate}
This proves that $c_{n_1,\dots,n_d}=0$ for all $n_1,\dots,n_d$ such that $n_1+\dots+n_d=n$. Therefore $\xi_n$ is of the form
$$\xi_n=\sum_{n_1+\dots+n_d\leq n-1}c_{n_1,\dots,n_d}X_1^{n_1}\dots X_d^{n_d}$$
But $\xi_n\in\mathcal{P}_n$. It follows that $\xi_n=0$.
\end{proof}


\section{The Jacobi sequences and linear change of basis of $V$}
In this section we fix a state $\varphi$ and we denote $\langle \ \cdot,\cdot\ \rangle_{\mathcal{P}_n}$ the pre-scalar product induced by $\varphi$ on $\mathcal{P}_n$ (see Lemma \ref{state}). Let
$$\left( \left(\Omega_n\right)_n,\:\:\left(\alpha_{.|n}\right)_n\right)$$
be the associated Jacobi sequences (see Theorem \ref{multi}).

Let $e=(e_j)_{1\leq j\leq d}$ be a linear basis of $V$. Recall that for $\eta\in\{+,0,-\}$
\begin{equation} \label{id1}
a^\eta_{v,e}=\sum_{j=1}^ dv_ja_j^\eta,\:\;a^\eta_{e_j,e}=a^\eta_j,
\end{equation}
where $v=\sum_{j=1}^dv_je_j$. Let $e'=(e'_j)_{1\leq j\leq d}$ be another basis of $V$. Denote
\begin{equation}\label{id2}
a_{v,e'}^{\eta}=\sum_{j=1}^du_ja_j^{\eta},\:\;a_{e'_j}^\eta=a^\eta_j,
\end{equation}
 where $v=\sum_{j=1}^du_je'_j.$

Note that from (\ref{6.1}) one has
\begin{eqnarray}\label{scalar1}
\langle v_1 \hat{\otimes}\dots\hat{\otimes}v_n,w_1 \hat{\otimes}\dots\hat{\otimes}w_n\rangle_{n,e}&=&\langle a_{v_n,e}^+\dots a_{v_1,e}^+\Phi, a_{w_n,e}^+\dots a_{w_1,e}^+\Phi\rangle_{\mathcal{P}_n}\\
&=&\langle v_n \hat{\otimes}\dots\hat{\otimes}v_1, \Omega_{n}^{(e)}w_n \hat{\otimes}\dots\hat{\otimes}w_1\rangle_{V^{\hat{\otimes}n}}\nonumber
\end{eqnarray}
for all $v_j,\;w_j\in V$ ($1\leq j\leq n$), where the scalar product on $V^{\hat{\otimes}n}$ is given by $\langle u^{\otimes n},v^{\otimes m}\rangle_{V^{\hat{\otimes}n}}=\delta_{m,n}\langle u,v\rangle_{V}^n$, with $\langle\;.\:,\:.\;\rangle_{V}$ is a fixed scalar product on $V$.

Similarly with respect to the basis $e'$ of $V$ one has
\begin{eqnarray}
\langle v_n \hat{\otimes}\dots\hat{\otimes}v_1,w_{n} \hat{\otimes}\dots\hat{\otimes}w_1\rangle_{n,e'}\!\!&=&\!\!\!\langle a_{v_n,e'}^{+}\dots a_{v_1,e'}^{+}\Phi, a_{w_n,e'}^{+}\dots a_{w_1,e'}^{+}\Phi\rangle_{\mathcal{P}_n}\label{scalar2}\\
&=&\!\!\!\langle v_n \hat{\otimes}\dots\hat{\otimes}v_1, \Omega_{n}^{(e')}w_n \hat{\otimes}\dots\hat{\otimes}w_1\rangle_{V^{\hat{\otimes}n}}\nonumber
\end{eqnarray}
Recall that the operator $\alpha_{.|_n}^{(e)}:V\rightarrow \mathcal{L}(V^{\hat{\otimes} n})$ is defined by
$$\alpha_{j|_n}^{(e)}=\alpha_{.|_n}^{(e)}e_j=\alpha_{e_j|_n}^{(e)}=U_{n,e}^{-1}a^{0}_{j,e|n}U_{n,e}$$
for all $j\in\{1,\dots,d\}$ where
\begin{eqnarray*}
U_{n,e}:V^{\hat{\otimes} n}&\rightarrow&\mathcal{P}_n\\
v_n\hat{\otimes}\dots\hat{\otimes}v_1&\mapsto&a^{+}_{v_n,e}\dots a^+_{v_1,e}\Phi
\end{eqnarray*}
\begin{theorem} \label{theo-fix}For all $n\in\mathbb{N}^*$, we have
\begin{equation}
\Omega_{n}^{(e)}=(R^{\otimes n})^*\Omega_n^{(e')}R^{\otimes n},\:\:\alpha_{e'_j|_n}^{(e')}=R^{\otimes n}\alpha_{e_j|_n}^{(e)}(R^{\otimes n})^{-1},
\end{equation}
for all $j\in\{1,\dots,d\}$ where $R=Pass(e,e')$, which is defined by $Re_i=e'_i$ for all $1\leq i\leq d$.
\end{theorem}
\begin{proof}
Since the map
\begin{eqnarray*}
V^n&\rightarrow& V^{\hat{\otimes}n}\\
(v_1,\dots,v_n)&\mapsto& v_1\hat{\otimes}\dots\hat{\otimes}v_n
\end{eqnarray*}
 is $n$-linear, it is sufficient to prove the theorem by taking $v_j$ and $w_j$ elements of the basis $e'$.

Let $e'_{j_1},\dots,e'_{j_n};\;e'_{l_1},\dots,e'_{l_n}$ be $2n$ elements of the linear basis $e'$. Then, from (\ref{scalar2}), one has
\begin{eqnarray}\label{id3}
\langle e'_{j_n} \hat{\otimes}\dots\hat{\otimes}e'_{j_1}, \Omega_{n}^{(e')}e'_{l_n} \hat{\otimes}\dots\hat{\otimes}e'_{l_1}\rangle_{V^{\hat{\otimes}n}}=
\langle a_{j_n}^{+}\dots a_{j_1}^{+}\Phi, a_{l_n}^{+}\dots a_{l_1}^{+}\Phi\rangle_{\mathcal{P}_n}
\end{eqnarray}
But, for all $1\leq j\leq d$
\begin{equation}\label{e}
a^+_j=a_{e_j}^+.
\end{equation}
Thus, identities (\ref{scalar2}) and (\ref{e}) give
\begin{eqnarray}\label{cc1}
\langle e'_{j_n} \hat{\otimes}\dots\hat{\otimes}e'_{j_1}, \Omega_{n}^{(e')}e'_{l_n} \hat{\otimes}\dots\hat{\otimes}e'_{l_1}\rangle_{V^{\hat{\otimes}n}}&=&
\langle a_{j_n}^{+}\dots a_{j_1}^{+}\Phi, a_{l_n}^{+}\dots a_{l_1}^{+}\Phi\rangle_{\mathcal{P}_n}\\
&=&\langle e_{j_n} \hat{\otimes}\dots\hat{\otimes}e_{j_1}, \Omega_{n}^{(e)}e_{l_n} \hat{\otimes}\dots\hat{\otimes}e_{l_1}\rangle_{V^{\hat{\otimes}n}}\nonumber
\end{eqnarray}
Note that
\begin{equation}\label{cc2}
Re_j=e'_j,\:1\leq j\leq d.
\end{equation}
Finally, from (\ref{cc1}) and (\ref{cc2}), one gets

$$\Omega_{n}^{(e)}=(R^{\otimes n})^*\Omega_n^{(e')}R^{\otimes n}.$$

Now, for all $j_1,\dots j_n\in\{1,\dots,d\}$, note that
\begin{eqnarray*}
U_{n,e'}e'_{j_n}\hat{\otimes}\dots\hat{\otimes}e'_{j_1}&=&a^+_{j_n}\dots a^+_{j_1}\Phi\\
U_{n,e}(R^{\otimes n})^{-1}e'_{j_n}\hat{\otimes}\dots\hat{\otimes}e'_{j_1}&=&a^+_{j_n}\dots a^+_{j_1}\Phi
\end{eqnarray*}
This implies that
$$U_{n,e'}=U_{n,e}(R^{\otimes n})^{-1}.$$
Then, one has
\begin{eqnarray*}
\alpha_{e'_j|_n}^{(e')} e'_{j_n}\hat{\otimes}\dots\hat{\otimes}e'_{j_1}&=&U_{n,e'}^{-1}a^{0}_{e'_j|n}U_{n,e'}  e'_{j_n}\hat{\otimes}\dots\hat{\otimes}e'_{j_1} \\
&=&U_{n,e'}^{-1}a^{0}_{j|n}a^+_{j_n}\dots a^+_{j_1}\Phi\\
&=&R^{\otimes n}U_{n,e}^{-1}a^0_{e_j|n}U_{n,e}e_{j_n}\hat{\otimes}\dots\hat{\otimes}e_{j_1}\\
&=&R^{\otimes n}\alpha_{e_j|_n}^{(e)}(R^{\otimes n})^{-1} e'_{j_n}\hat{\otimes}\dots\hat{\otimes}e'_{j_1}
\end{eqnarray*}
This ends the proof.
\end{proof}

\section{Connection between the classical and multi-dimensional Favard Lemmas in the case of $d=1$}

Let $d=1$, $\mu$ be a probability measure on $\mathbb{R}$ with finite moments of any order. Let $(\alpha_n,\beta_n,P_n)_n$ be the classical Favard Lemma sequences associated to $\mu$. Then the Jacobi relation is given by
\begin{eqnarray*}
XP_n&=&\beta_nP_{n+1}+\alpha_n P_n+ \beta_{n-1}P_{n-1}\\
P_{-1}&=&0,
\end{eqnarray*}
where $(P_n)_n$ is a family of orthonormal polynomials (unitary) with respect to the scalar product induced by $\mu$.

In this case the CAP operators take the form
\begin{eqnarray*}
a^+_1{|_n}P_n&=&\beta_nP_{n+1}\\
a^-_1{|_n}P_n&=&\beta_{n-1}P_{n-1}\\
a^0_1{|_n}P_n&=&\alpha_n P_n
\end{eqnarray*}
It is clear that $(a^0_1|_n)^*=a^0_1|_n$ and $a_1^-|_n=(a^+_1|_{n-1})^*$ with respect to the scalar product induced by $\mu$ on $\mathcal{P}_{n]}$. Put $V=\mathbb{C}$ equipped with the canonical scalar product. If  $e=(e_1)$ and $e'=(e'_1)$ of $\mathbb{C}$ ($d=1$) and $Re_1=e'_1=ae_1$, then from Theorem \ref{theo-fix} the positive scalars $\Omega_n^{(e)}$ and $\Omega_n^{(e')}$ (because $d=1$) satisfy
$$\Omega_n^{(e)}=(R^*)^{\otimes n}\Omega_n^{(e')}R^{\otimes n}$$
But $R^*v=\bar{a}v$. This gives
$$\Omega_n^{(e)}=|a|^{2n}\Omega_n^{(e')}$$
Since $R^{-1}e'_1=\frac{1}{a}e_1$, it is clear that $\alpha_{e_1}=\alpha_{e'_1}$.


Now, suppose that $a^+_1=a_{e_1}^+$ where $e=(e_1)$ is the canonical basis of $\mathbb{C}$ and denote $\Omega_n^{(e)}=\Omega_n$. Then $(e_1^{\hat\otimes n})$ is a basis of $\mathbb C^{\hat\otimes n}\simeq \mathbb{C}$. Moreover, one has
\begin{eqnarray}\label{8}
\Omega_n&=&\langle e_1\hat{\otimes}\dots\hat{\otimes}e_1,\Omega_n e_1\hat{\otimes}\dots\hat{\otimes}e_1\rangle_{\mathbb{C}^{\hat{\otimes} n}}\nonumber\\
&=&\langle a^+_1\dots a^+_1\Phi, a^+_1\dots a^+_1\Phi\rangle_{\mathcal{P}_n}
\end{eqnarray}
Note that $\Phi=1_\mathcal{P}=1=P_0$. This yields
\begin{eqnarray*}
a^+_1\Phi&=&\beta_0P_1\\
a^+_1a^+_1\Phi&=&\beta_1\beta_0P_2\\
&\vdots&\\
(a^+_1)^n\Phi&=&\beta_{n-1}\dots \beta_1\beta_0P_n
\end{eqnarray*}
Using this and identity (\ref{8}) it follows that
$$\Omega_n=\beta_0^2\dots \beta_{n-1}^2$$



Now let $v_1,\dots,v_n\in\mathbb{C}$. We have
\begin{eqnarray*}
\alpha_{1|_n}v_1\hat{\otimes}\dots\hat{\otimes}v_n&=&(\alpha_{.|_n}e_1)v_1\hat{\otimes}\dots\hat{\otimes}v_n\\
&=&U_n^{-1}a^0_{1|n}U_nv_1\hat{\otimes}\dots\hat{\otimes}v_n\\
&=&U_n^{-1}a^0_{1|n}a^+_{v_n}\dots a^+_{v_1}\Phi\\
&=&(v_n\dots v_1)U_n^{-1}a^0_{1|n}a^+_{1}\dots a^+_{1}\Phi\\
&=&(v_n\dots v_1)(\beta_{n-1}\dots \beta_0)U_n^{-1}a^0_{1|n}P_n\\
&=&\alpha_n(v_n\dots v_1)(\beta_{n-1}\dots \beta_0)U_n^{-1}P_n\\
&=&\alpha_nv_1\hat{\otimes}\dots\hat{\otimes}v_n
\end{eqnarray*}
which proves that $\alpha_{1|_n}=\alpha_n$.
\section{Product  probability measures}
Recall that if $\varphi$ is a state on $\mathcal{P}=\mathbb{C}((X_j)_{1\leq j\leq d})$ and if we denote $\langle.,\,.\rangle_{\mathcal{P}_n}$ the pre-scalar product induced by $\varphi$ on $\mathcal{P}_n$ then one has
\begin{equation}\label{1}
\langle v_n\hat{\otimes}\dots\hat{\otimes}v_1,\Omega_n u_n\hat{\otimes}\dots\hat{\otimes}u_1\rangle_{V^{\hat{\otimes} n}}=\langle a^+_{v_n}\dots a^+_{v_1}\Phi,a^+_{u_n}\dots a^+_{u_1}\Phi\rangle_{\mathcal{P}_n}.
\end{equation}
In the following, we will take $V=\mathbb{C}^d$ equipped with the canonical scalar product.

Let $\mu_1,\dots,\mu_d$ be $d$ probability measures on $\mathbb{R}$ with finite moments of any order. By the classical Favard Lemma there exist sequences $(\alpha_{k,n},\beta_{k,n},P_{k,n})_n$ $(1\leq k\leq d)$ such that:
\begin{enumerate}
\item[(i)] $(\beta_{k,n})_n$ is a sequence of positive real scalars,
\item[ii)] $(\alpha_{k,n})_n$ is a sequence of real scalars,
\item[iii)] $(P_{k,n})_n$ is a family of orthonormal polynomials with respect to the scalar product induced by the probability measure $\mu_k$ which satisfies:
\begin{eqnarray}
XP_{k,n}&=&\beta_{k,n}P_{k,n+1}+\alpha_{k,n}P_{k,n}+\beta_{k,n-1}P_{k,n-1}\label{2}\\
P_{k,-1}&=&0,\:\:P_{k,0}=1\nonumber
\end{eqnarray}
\end{enumerate}

{\bf Remark}: Let
$$\mu=\mu_1\otimes\dots\otimes\mu_j\otimes\dots\otimes\mu_d$$
be the probability measure on $\mathbb{R}^d$. It is clear that
\begin{equation}\label{3}
P_{\bar{n}}=P_{1,n_1}\otimes\dots \otimes P_{j,n_j}\otimes\dots\otimes P_{d,n_d},
\end{equation}
where $\bar{n}= (n_1,\dots,n_d)$ and $j$ indicates the $j-th$ variable $X_j$, is a family of orthogonal polynomials with respect to the scalar product induced by $\mu$. Note that $P_{\bar{n}}$ are polynomials in $d$-variables $X_1,\dots,X_d$ and $P_{j,n_j}$ is understood as a polynomial acting only the $j-th$ variable $X_j$.

Define the CAP operators with respect to the canonical basis $(e_j)_{1\leq j\leq d}$ of $\mathbb{C}^d$ as follows:
\begin{eqnarray}\label{4}
a^+_{k,n}P_{k,n}&=&a^+_k|_n P_{k,n}=a^+_{e_k|n}P_{k,n} =\beta_{k,n}P_{k,n+1}\nonumber\\
a^-_{k,n}P_{k,n}&=&a^-_k|_nP_{k,n}=a^-_{e_k|n}P_{k,n} =\beta_{k,n-1}P_{k,n-1}\\
a^0_{k,n}P_{k,n}&=&a^0_k|_nP_{k,n}=a^0_{e_k|n}P_{k,n}=\alpha_{k,n}P_{k,n}\nonumber
\end{eqnarray}
It is straightforward to show that $a^0_{k,n}$ is a self-adjoint operator and $a^-_{k,n}=(a^+_{k,n-1})^*$ with respect to the scalar product induced by the probability measure $\mu_k$ on $\mathcal{P}_{k,n]}$ (it is sufficient to verify these identity on the orthonormal basis $(P_{k,n})_n$).\\

{\bf Remarks}: Note that

\begin{enumerate}
\item[1)] $a^\varepsilon_{k,n}=a^\varepsilon_{e_k,n}$
\item[2)] $a^\varepsilon_{k}=\sum_n a^\varepsilon_{k,n}$
\item[3)] The CAP operators (\ref{4}) act on the tensor product (\ref{3}) as follows:
$$a^\varepsilon_{k,n}=I\otimes \dots\otimes a^\varepsilon_{k,n}\otimes I\dots\otimes I$$
This means that $a^\varepsilon_{k,n}$ acts only in the $k$-variable $``X_k''$.
\end{enumerate}

In the following, our purpose is to give the explicit form of the Jacobi sequences associated to the probability measure $\mu$.

Define the equivalence relation $\mathcal{R}$ on $\{1,\dots,d\}^n$ by
$$(i_1,\dots,i_n)\mathcal{R}(j_1,\dots,j_n)$$
if and only if there exists a permutation $\pi\in S_n$ such that $i_k=j_{\pi(k)}$ for all $k\in\{1,\dots,n\}$. Denote the $R$-equivalence class of an element $(i_1,\dots,i_n)\in\{1,\dots,d\}^n$ by $cl(i_1,\dots,i_n)$. Put
\begin{eqnarray}
\mathcal{A}_n&:=&\{\bar{j}_n=cl\Big((j_1,\dots, j_n)\Big);\:j_k\in\{1,\dots,d\}\}\label{Am}\\
e_{\bar{j}_n}&:=&e_{j_1}\hat{\otimes}\dots\hat{\otimes}e_{j_n}\nonumber
\end{eqnarray}
\begin{lemma}
The family $\mathcal{B}=(e_{\bar{j}_n})_{\bar{j}_n\in\mathcal{A}_n}$ is a basis of $(\mathbb{C}^d)^{\hat{\otimes}n}$ and
\begin{eqnarray*}\label{card}
\dim(\mathbb{C}^d)^{\hat{\otimes}n}=\sharp \mathcal{A}_n=\left(
\begin{array}{lcc}
n+d-1\\
d-1
\end{array}
\right)
\end{eqnarray*}
\end{lemma}
\begin{proof}
It is an easy to show that $\mathcal{B}$ is a basis of $(\mathbb{C}^d)^{\hat{\otimes}n}$. Moreover, it is well known that
$$\dim(\mathbb{C}^d)^{\hat{\otimes}n}=\left(
\begin{array}{lcc}
n+d-1\\
d-1
\end{array}
\right)$$
This completes the proof .
\end{proof}

\begin{theorem} In the basis $\mathcal{B}$, we have
\begin{eqnarray*}
\Omega_{n,\bar{i}_n\bar{j}_n}&=&\delta_{\bar{i}_n\bar{j}_n}\Big(\prod_{k_1=0}^{m_1-1}\beta_{1,k_1}\Big)^2\dots \Big(\prod_{k_r=0}^{m_r-1}\beta_{r,k_r}\Big)^2\dots\Big(\prod_{k_d=0}^{m_d-1}\beta_{d,k_d}\Big)^2\\
\alpha_{e_l|_n}e_{\bar{i}_n}&=&\alpha_{l,m_l}e_{\bar{i}_n}
\end{eqnarray*}
for all $\bar{i}_n=cl(i_1,\dots,i_n);\;\bar{j}_n=cl(j_1,\dots,j_n)\in\mathcal{A}_n$ and where $m_l=\sharp\Big(\{i_k=l;k=1,\dots,n\}\Big)$ ($1\leq l\leq d$) with the convention $\prod_{k=0}^{-1}\beta_{l,k}=1$ (this convention is used when $m_l=0$).
\end{theorem}
\begin{proof}
Let $\bar{i}_n=cl(i_1,\dots,i_n);\;\bar{j}_n=cl(j_1,\dots,j_n)\in\mathcal{A}_n$. For all $1\leq l\leq d$ put
\begin{eqnarray}\label{(*)}
m_l&=&\sharp\Big(\{i_k=l;k=1,\dots,n\}\Big)\nonumber\\
n_l&=&\sharp\Big(\{j_k=l;k=1,\dots,n\}\Big)
\end{eqnarray}
Then one has
\begin{eqnarray}
\Omega_{n,\bar{i}_n\bar{j}_n}&=&\langle e_{\bar{i}_n},\Omega_n e_{\bar{j}_n}\rangle_{(\mathbb{C}^d)^{\hat{\otimes}n}}\nonumber\\
&=&\langle e_{i_1}\hat{\otimes}\dots\hat{\otimes}e_{i_n},\Omega_ne_{j_1}\hat{\otimes}\dots\hat{\otimes}e_{j_n}\rangle_{(\mathbb{C}^d)^{\hat{\otimes}n}}\nonumber\\
&=&\langle a_{i_1}^+\dots a_{i_n}^+\Phi,a_{j_1}^+\dots a_{j_n}^+\Phi \rangle_{\mu_1\otimes\dots\otimes \mu_d}\nonumber\\
&=&\langle (a^+_1)^{m_1}\dots (a^+_d)^{m_d}\Phi,(a^+_1)^{n_1}\dots (a^+_d)^{n_d}\Phi\rangle_{\mu_1\otimes\dots\otimes \mu_d}\label{00}
\end{eqnarray}
But for all $l\in\{1,\dots,d\}$ the creator $a^+_l$ acts only on the $l-th$ variable $X_l$. Remember that
$$P_{l,0}=\Phi=1,\:\forall l\in\{1,\dots,d\}$$
It follows that
\begin{eqnarray}
(a^+_1)^{m_1}\dots (a^+_d)^{m_d}\Phi&=&(a^+_1)^{m_1}\dots (a^+_d)^{m_d}1\nonumber\\
&=&\Big[(a^+_1)^{m_1}1\Big]\dots\Big[(a^+_d)^{m_d}1\Big]\label{11}\\
&=&\Big[(a^+_1)^{m_1}P_{1,0}\Big]\dots\Big[(a^+_d)^{m_d}P_{d,0}\Big]\nonumber
\end{eqnarray}
Identity (\ref{11}) is true because $(a^+_d)^{m_d}1$ is a polynomial in the variable $X_d$ which is of the form $(a^+_d)^{m_d}1=Q(X_d)=\prod_{k_d=0}^{m_d-1}\beta_{d,k_d}P_{d,m_d}$, so this gives
$$(a^+_{d-1})^{m_{d-1}}(a^+_d)^{m_d}1=(a^+_{d-1})^{m_{d-1}}Q(X_d)=Q(X_d) (a^+_{d-1})^{m_{d-1}}1$$
Then from (\ref{11}) one has
\begin{equation}\label{22}
(a^+_1)^{m_1}\dots (a^+_d)^{m_d}\Phi=\Big(\prod_{k_1=0}^{m_1-1}\beta_{1,k_1}\Big)\dots\Big(\prod_{k_d=0}^{m_d-1}\beta_{d,k_d}\Big)P_{1,m_1}\dots P_{d,m_d}
\end{equation}
\begin{enumerate}
\item[-] If $\bar{i}_n\neq \bar{j}_n$, then there are two cases:
$$\{i_1,\dots,i_n\}\neq \{j_1,\dots,j_n\}$$
 or there exists $l\in\{1,\dots,d\}$ such that
$$\sharp\Big(\{i_k=l,k=1,\dots,n\}\Big)\neq \sharp\Big(\{j_k=l,k=1,\dots,n\}\Big)$$
\begin{enumerate}
\item[(i)] First case: if $\{i_1,\dots,i_n\}\neq \{j_1,\dots,j_n\}$, then there exists $l\in\{1,\dots, d\}$ such that $l\in \{i_1,\dots,i_n\}$ and $l\notin\{j_1,\dots,j_n\}$ or the inverse. Without loss of generality, suppose that $l=1\in\{i_1,\dots,i_n\}$ and $l\notin\{j_1,\dots,j_n\}$. Then from identity (\ref{(*)}) it follows that $n_1=0$. Hence identity (\ref{00}) implies that
\begin{eqnarray*}
\Omega_{n,\bar{i}_n\bar{j}_n}&=&\langle (a^+_1)^{m_1}\dots (a^+_d)^{m_d}\Phi,(a^+_2)^{n_2}\dots (a^+_d)^{n_d}\Phi\rangle_{\mu_1\otimes\dots\otimes \mu_d}\\
&=&\Big(\prod_{k_1=0}^{m_1-1}\beta_{1,k_1}\Big)\langle P_{1,m_1},1\rangle_{\mu_1}\\
&&\times\langle (a^+_2)^{m_2}\dots (a^+_d)^{m_d}\Phi,(a^+_2)^{n_2}\dots (a^+_d)^{n_d}\Phi\rangle_{\mu_2\otimes\dots\otimes \mu_d}\\
&=&0
\end{eqnarray*}
because $m_1\geq1$ and $(P_{1,m})_m$ is an orthogonal family with respect to the scalar product induced by $\mu_1$.
\item[(ii)] Second case: if there exists $l\in\{1,\dots,d\}$ such that
$$\sharp\Big(\{i_k=l,k=1,\dots,n\}\Big)\neq \sharp\Big(\{j_k=l,k=1,\dots,n\}\Big)$$
Without loss of generality, suppose that $l=1$, i.e $m_1\neq n_1$. Then from (\ref{00}) one has:
\begin{eqnarray*}
\Omega_{n,\bar{i}_n\bar{j}_n}&=&\langle (a^+_1)^{m_1}\dots (a^+_d)^{m_d}\Phi,(a^+_1)^{n_1}\dots (a^+_d)^{n_d}\Phi\rangle_{\mu_1\otimes\dots\otimes \mu_d}\\
&=&\langle (a^+_1)^{m_1}\Phi, (a^+_1)^{n_1}\Phi\rangle_{\mu_1}\\
&&\times\langle (a^+_2)^{m_2}\dots (a^+_d)^{m_d}\Phi,(a^+_2)^{n_2}\dots (a^+_d)^{n_d}\Phi\rangle_{\mu_2\otimes\dots\otimes \mu_d}\\
&=&\Big(\prod_{k_1=0}^{m_1-1}\beta_{d,k_1}\Big)\Big(\prod_{l_1=0}^{n_1-1}\beta_{1,l_1}\Big)\langle P_{1,m_1}, P_{1,n_1}\rangle_{\mu_1}\\
&&\times \langle (a^+_2)^{m_2}\dots (a^+_d)^{m_d}\Phi,(a^+_2)^{n_2}\dots (a^+_d)^{n_d}\Phi\rangle_{\mu_2\otimes\dots\otimes \mu_d}\\
&=&0
\end{eqnarray*}
because $P_{1,n_1}$ is orthogonal to $P_{1,m_1}$ ($n_1\neq m_1$).
\end{enumerate}
\item[-] If $\bar{i}_n=\bar{j}_n$ then one has
\begin{eqnarray*}
\Omega_{n,\bar{i}_n\bar{i}_n}&=&\langle (a^+_1)^{m_1}\dots (a^+_d)^{m_d}\Phi,(a^+_1)^{m_1}\dots (a^+_d)^{m_d}\Phi\rangle_{\mu_1\otimes\dots\otimes \mu_d}\\
&=&\langle (a^+_1)^{m_1}\Phi,(a^+_1)^{m_1}\Phi\rangle_{\mu_1}\dots\langle (a^+_d)^{m_d}\Phi,(a^+_d)^{m_d}\Phi\rangle_{\mu_d}\\
&=&\Big(\prod_{k_1=0}^{m_1-1}\beta_{1,k_1}\Big)^2\dots\Big(\prod_{k_d=0}^{m_d-1}\beta_{d,k_d}\Big)^2
\end{eqnarray*}
\end{enumerate}

Now let $\bar{i}_n=cl\{i_1,\dots,i_n\}\in\mathcal{A}_n$. Recall that
$$U_n e_{\bar{i}_n}=U_ne_{i_1}\hat{\otimes}\dots\hat{\otimes}e_{i_n}=a^+_{i_1}\dots a^+_{i_n}\Phi$$
Then for all $l\in\{1,\dots,d\}$ one has
\begin{eqnarray*}
\alpha_{e_l|_n} e_{\bar{i}_n}&=&U_n^{-1}a^0_{l|_n}U_ne_{\bar{i}_n}\\
&=&U_n^{-1}a^0_{l|_n}a^+_{i_1}\dots a^+_{i_n}\Phi
\end{eqnarray*}
Let $m_l=\sharp\{i_k=l,\:k=1,\dots,n\}$. Then there are two cases:
\begin{enumerate}
\item[-] First case: if $m_l>0$ it follows that
\begin{eqnarray*}
a^0_{l|_n}a^+_{i_1}\dots a^+_{i_n}\Phi&=&a^0_{l|_n}\Big(\prod_{i_k\neq l}a^+_{i_k}\Big)(a^+_l)^{m_l}\Phi\\
&=&a^0_{l|_n}\Big(\prod_{h=0}^{m_l-1}\beta_{l,h}\Big)P_{l,m_l}\Big(\prod_{i_k\neq l}a^+_{i_k}\Big)\Phi\\
&=&\Big(\prod_{h=0}^{m_l-1}\beta_{l,h}\Big)\alpha_{l,m_l}P_{l,m_l}\Big(\prod_{i_k\neq l}a^+_{i_k}\Big)\Phi\\
&=&\alpha_{l,m_l}a^+_{i_1}\dots a^+_{i_n}\Phi
\end{eqnarray*}
This implies that
$$\alpha_{e_l|_n} e_{\bar{i}_n}=\alpha_{l,m_l}U_n^{-1}a^+_{i_1}\dots a^+_{i_n}\Phi=\alpha_{l,m_l}e_{\bar{i}_n}$$
\item[-] Second case: if $m_l=0$ then one has
\begin{eqnarray*}
\alpha_{e_l|_n} e_{\bar{i}_n}&=&U_n^{-1}a^0_{l|_n}a^+_{i_1}\dots a^+_{i_n}\Phi\\
&=&U_n^{-1}a^+_{i_1}\dots a^+_{i_n}a^0_{l|_n}\Phi\\
&=&\alpha_{l,0}U_n^{-1}a^+_{i_1}\dots a^+_{i_n}\Phi=\alpha_{l,0}e_{\bar{i}_n}
\end{eqnarray*}
\end{enumerate}
\end{proof}

\bigskip
{\bf\large Acknowledgments}\bigskip

The authors gratefully acknowledge stimulating discussions with Luigi Accardi and would like to thank him for reading the paper and for interesting comments.

\end{document}